\documentclass{mcom-l}

\newtheorem{theorem}{Theorem}[section]
\newtheorem{lemma}[theorem]{Lemma}

\theoremstyle{definition}

\theoremstyle{remark}

\numberwithin{equation}{section}

\usepackage{graphicx} 
\usepackage{extarrows}
\usepackage{latexsym}
\usepackage{amsmath}
\usepackage{amssymb}
\usepackage{amsfonts}
\usepackage{verbatim}
\usepackage{mathrsfs}
\usepackage[modulo]{lineno} % 行数
\usepackage[latin1]{inputenc}
\usepackage{color}
\usepackage{xcolor}
\usepackage[colorlinks,citecolor=blue,urlcolor=blue]{hyperref}
\usepackage{soul}% 高亮

\newtheorem{rk}{Remark}[section]
\newtheorem{ap}{Assumption}[section]

\newtheorem{prop}{Proposition}[section]
\newtheorem{lm}{Lemma}[section]

\newcommand{\tr}{\top}

\newcommand{\ee}{\mathbb E}

\newcommand{\pp}{\mathbb P}
\newcommand{\nn}{\mathbb N}
\newcommand{\rr}{\mathbb R}

\newcommand{\BB}{\mathcal B}
\newcommand{\CC}{\mathcal C}
\newcommand{\LL}{\mathcal L}

\newcommand{\PP}{\mathcal P}

\newcommand{\OOO}{\mathscr O}
\newcommand{\FFF}{\mathscr F}

\newcommand{\<}{\langle}
\renewcommand{\>}{\rangle}
\allowdisplaybreaks \allowdisplaybreaks[4]

\newcommand{\norm}[1]{\left\lVert #1 \right\rVert}

\begin{document}

\title[Numerical Ergodicity of SACE driven by Multiplicative White Noise]
{Numerical Ergodicity of Stochastic Allen--Cahn Equation driven by Multiplicative White Noise} 

%    Information for the first author
\author{Zhihui LIU}
\address{Department of Mathematics \& National Center for Applied Mathematics Shenzhen (NCAMS) \& Shenzhen International Center for Mathematics, Southern University of Science and Technology, Shenzhen 518055, China}
% \curraddr{}
\email{liuzh3@sustech.edu.cn}
\iffalse
\author{}
\address{}
% \curraddr{}
\email{}
\fi
\thanks{The author is supported by the National Natural Science Foundation of China, No. 12101296, Basic and Applied Basic Research Foundation of Guangdong Province, No. 2024A1515012348, and Shenzhen Basic Research Special Project (Natural Science Foundation) Basic Research (General Project), No. JCYJ20220530112814033 and JCYJ20240813094919026.} 

%    General info



\begin{abstract}
We establish the unique ergodicity of a fully discrete scheme for monotone SPDEs with polynomial growth drift and bounded diffusion coefficients driven by multiplicative white noise.
The main ingredient of our method depends on the satisfaction of a Lyapunov condition followed by a uniform moments' estimate, combined with the regularity property for the full discretization.
We transform the original stochastic equation into an equivalent random equation where the discrete stochastic convolutions are uniformly controlled to derive the desired uniform moments' estimate.
Applying the main result to the stochastic Allen--Cahn equation driven by multiplicative white noise indicates that this full discretization is uniquely ergodic for any interface thickness.
Numerical experiments validate our theoretical results. 
\end{abstract}

\maketitle

\section{Introduction}

The invariant measure and ergodicity, as a significant long-time behavior of Markov processes generated by stochastic ordinary and partial differential equations (SODEs and SPDEs, respectively), characterize the identity of temporal average and spatial average, which has many applications in quantum mechanics, fluid dynamics, financial mathematics, and many other fields \cite{DZ96, HW19}. 
As everyone knows, the explicit expression of the invariant measure for a stochastic nonlinear system is rarely available.
For this reason, it motivated and fascinated a lot of investigations in recent decades for designing numerical algorithms that can inherit the ergodicity of the original system.

There have been some developments in the construction and analysis of numerical algorithms for the invariant measures and ergodic limits of dissipative SODEs, Lipschitz SPDEs, or SPDEs with super-linear growth coefficients driven by trace-class noise. See, e.g., \cite{LMYY18, LMW23, MSH02} and references therein for numerical ergodicity of dissipative SODEs with or without Markovian switching, \cite{Bre14, BV16, CGW20, CHS21, HM06} for approximating the invariant measures of parabolic SPDEs driven by additive noise, \cite{Liu23, LL24} for the unique ergodicity of the drift-implicit Euler Galerkin (DIEG) scheme of monotone SPDEs with polynomial growth coefficients driven by multiplicative trace-class noise.

In the settings of the infinite-dimensional case, we note that most of the existing literature focuses on the numerical ergodicity of Lipschitz SPDEs driven by additive white noise or monotone SPDEs driven by trace-class noise; the case of super-linear SPDEs driven by multiplicative white noise is more subtle and challenging. 
Our main aim is to indicate that the widely studied DIEG scheme (see \eqref{die-g}) applied to second-order monotone SPDEs with polynomial growth coefficients driven by nondegenerate multiplicative white noise is uniquely ergodic (see Theorem \ref{tm-spde}). 
Applying this result to the 1D stochastic Allen--Cahn equation driven by nondegenerate multiplicative white noise indicates that its DIEG scheme is uniquely ergodic for any interface thickness (see Theorem \ref{tm-ac}).

The paper is organized as follows.
Section \ref{sec2} gives the principal assumptions on the considered SPDE.
In this part, we show the unique solvability and required properties of the DIEG scheme.
The Lyapunov structure and regularity property of the DIEG scheme with application to the stochastic Allen--Cahn equation are explored in Section \ref{sec3}. 
The theoretical results are verified by numerical experiments in Section \ref{sec4}.

\section{Preliminaries}
\label{sec2}

This section presents the main assumptions used throughout the paper, the solvability, and the properties needed for the full discretization to be considered. 
We also give some preliminaries on invariant measure and ergodicity of Markov chains.

\subsection{Main Assumptions}

Denote by $\|\cdot\|$ and $\<\cdot, \cdot\>$ blue the norm and inner product, respectively, in $H:=L^2(0, 1)$ equipped with the Borel $\sigma$-algebra $\BB(H)$.      
For $\theta \in [-1, 1]$, we use $(\dot H^\theta=\dot H^\theta(0, 1), \|\cdot\|_\theta)$ to denote the usual Sobolev interpolation spaces; the dual between $\dot H^1$ and $\dot H^{-1}$ is denoted by $_1\<\cdot, \cdot\>_{-1}$. 
We use $\CC_b(H)$ and $(\LL_2^0, \|\cdot\|_{\LL_2^0})$ to denote the space of bounded, continuous functions and Hilbert--Schmidt operators on $H$, respectively.      

Let $W$ be an $H$-valued cylindrical Wiener process on a complete filtered probability space $(\Omega, \FFF, (\FFF(t))_{t\geq 0}, \pp)$, i.e., there exists an orthonormal basis $\{g_k\}_{k=1}^\infty$ of $H$ and a sequence of mutually independent Brownian motions $\{\beta_k\}_{k=1}^\infty$ such that (see \cite[Proposition 2.1.10]{LR15})
\begin{align*}
W(t, \xi)
=\sum_{k\in \nn_+} g_k(\xi) \beta_k(t),
\quad (t, \xi) \in \rr_+ \times (0, 1).
\end{align*}
In the distribution sense, the spatial derivative of $W$ is the so-called (space-time) white noise.   
  
Denote by $\Delta$ the Dirichlet Laplacian operator on $H$.
Then $-\Delta$ possesses a sequence of blue positive eigenvalues $\{\lambda_k\}_{k \in \nn_+}$ in an increasing order corresponding to the eigenvectors $\{e_k\}_{k \in \nn_+}$ which vanish on the boundary of $(0, 1)$, respectively:
\begin{align} \label{Delta}
-\Delta e_k=\lambda_k e_k, \quad k \in \nn_+.
\end{align} 
It is clear that the following Poincar\'e inequality holds (with $\lambda_1=\pi^2$):
\begin{align} \label{poin}
\|\nabla x\|^2 \ge \lambda_1 \|x\|^2, \quad x \in \dot H^1.
\end{align} 
 
Let us consider the following second-order parabolic SPDE driven by an $H$-valued cylindrical Wiener process $W$:
\begin{align}\label{see-fg} 
& {\rm d} X(t, \xi)  
=(\Delta X(t, \xi)+f(X(t, \xi))) {\rm d}t
+g(X(t, \xi)) {\rm d}W(t, \xi), 
\quad (t, \xi) \in \rr_+ \times \OOO, 
\end{align}
under (homogenous) Dirichlet boundary condition (DBC) $X(t, \xi)=0$, $(t, \xi) \in \rr_+ \times \partial \OOO$, with the initial datum $X_0(\xi):=X(0, \xi)$, $\xi \in \OOO$.
Here, $f$ is assumed to be monotone-type with polynomial growth, and $g$ is a continuous, bounded, and invertible function (see Assumptions \ref{ap-f} and  \ref{ap-g}).

It is known that Eq. \eqref{see-fg} driven by white noise possesses a random field solution only in 1D; thus, we restrict our investigation to the 1D physical domain $\OOO=(0, 1)$. 
We note that Eq. \eqref{see-fg} includes the following stochastic Allen--Cahn equation under DBC, arising from phase transition in materials science by stochastic perturbation, as a special case:
\begin{align} \label{ac}
{\rm d} X(t, \xi)=  \Delta X(t, \xi)  {\rm d}t + \epsilon^{-2} (X(t, \xi)-X(t, \xi)^3) {\rm d}t
+ g(X(t, \xi)) {\rm d}W(t, \xi), 
\end{align}
where the positive index $\epsilon \ll 1$ is the interface thickness; see, e.g., \cite{BGJK23, BCH19, Liu23, LL24, LQ20, LQ21} and references therein.

Our main conditions on the coefficients of Eq. \eqref{see-fg} are the following two assumptions.

\begin{ap} \label{ap-f}
There exist scalars $K_i \in \rr$, $i=1,2,3,4,5$, and $q \ge 1$ such that 
\begin{align} 
(f(\xi)-f(\eta)) (\xi-\eta) & \le K_1 (\xi-\eta)^2,
\quad \xi, \eta \in \rr, \label{f-mon} \\
 f(\xi) \xi & \le  K_2 |\xi|^2 + K_3,
\quad \xi \in \rr,  \label{f-coe} \\
 |f(\xi)| &  \le K_4 |\xi|^q+K_5,\quad \xi \in \rr. \label{f-gro}
\end{align}
\end{ap}

\begin{rk}
It is clear that the corresponding function $f(\xi):=\epsilon^{-2} (\xi-\xi^3)$, $\xi \in \rr$, in the stochastic Allen--Cahn equation \eqref{ac} satisfies Assumption \ref{ap-f}; see \cite[Example 2.1]{LQ21}.
\end{rk}

Define the Nemytskii operator $F: \dot H^1 \rightarrow \dot H^{-1}$ associated with $f$ by
\begin{align}  \label{df-F}
F(x)(\xi):=f(x(\xi)), & \quad x \in \dot H^1,\ \xi \in (0, 1).
\end{align} 
Then the monotone condition \eqref{f-mon} and the coercive condition \eqref{f-coe} yield that the operator $F$ defined in \eqref{df-F} satisfies 
\begin{align} 
_{1}\<x-y, F(x)-F(y)\>_{-1} & \le K_1 \|x-y\|^2, \quad x,y \in \dot H^1, \label{F-mon}  \\
_{1}\<x, F(x)\>_{-1} & \le  K_2 \|x\|^2+K_3, \quad x \in \dot H^1. \label{F-coe} 
\end{align} 
 
\begin{ap} \label{ap-g}
$g: \rr \to \rr$ is continuous, invertible, and bounded, i.e., there exists a nonnegative constant $K_6$ such that 
\begin{align}   
0 \neq |g(\xi)| & \le K_6, \quad \xi \in \rr. \label{g-bou} 
\end{align} 
\end{ap}

As in \eqref{df-F}, one can define the Nemytskii operator $G: H \rightarrow \LL_2^0$ associated with $g$ by
\begin{align}   \label{df-G}
G(x) g_k(\xi):=g(x(\xi)) g_k(\xi), & \quad x \in H,~ k \in \nn,~ \xi \in (0, 1). 
\end{align} 
Then Eq. \eqref{see-fg} is equivalent to the following infinite-dimensional stochastic evolution equation:
\begin{align} \label{see}
{\rm d}X(t)=(\Delta X(t)+F(X(t))) {\rm d}t+G(X(t)) {\rm d}W, ~~  t \ge 0;
\quad X(0)=X_0.
\end{align}

\subsection{DIEG scheme and solvability}

To introduce the fully discrete scheme, let $N \in \nn_+$ and $V_N$ be the space spanned by the first $N$-eigenvectors of $\Delta$:
\begin{align*}
V_N:={\rm span}\{e_1, e_2, \cdots, e_N\},
\quad N \in \nn_+.
\end{align*}
Define the spectral Galerkin approximate Laplacian operator $\Delta_N: V_N \rightarrow V_N$ and the generalized orthogonal projection operator $\PP_N: \dot H^{-1} \rightarrow V_N$, respectively, as 
\begin{align*}  
\<\Delta_N u^N, v_N\> & =-\<\nabla u^N, \nabla v_N\>,
\quad u^N, v_N \in V_N,  \\
\<\PP_N u, v_N\> & =_1\<v_N, u\>_{-1},
\quad u \in \dot H^{-1},\ v_N \in V_N. 
\end{align*} 
Then the DIEG scheme of Eq. \eqref{see} is to find a $V_N$-valued discrete process $\{X^N_j:\ j \in \nn\}$ such that
\begin{align}\label{die-g} \tag{DIEG} 
&X^N_{j+1}
=X^N_j+\tau \Delta_N X^N_{j+1}
+\tau \PP_N F(X^N_{j+1})
+\PP_N G(X^N_j) \delta_j W,  
\end{align} 
starting from the $V_N$-valued r.v. $X_0^N$ (usually, one takes $X_0^N=\PP_N X_0$ with $X_0$ being the initial datum of Eq. \eqref{see}), where $\delta_j W=W(t_{j+1})-W(t_j)$,  $j \in \nn$.
This fully discrete scheme and its Galerkin finite element version have been widely studied; see, e.g., \cite{CHS21, Liu23, Liu22, LL24, LQ21}. 

It is clear that the DIEG scheme \eqref{die-g} is equivalent to the scheme
\begin{align}\label{full+}
X^N_{j+1}=S_{N,\tau} X^N_j+\tau S_{N,\tau} \PP_N F(X^N_{j+1})
+ S_{N,\tau} \PP_N G(X^N_j) \delta_j W,
\quad j \in \nn,
\end{align}
where $S_{N,\tau}:=({\rm Id}-\tau \Delta_N)^{-1}$ is a space-time approximation of the continuous semigroup $\{S(t)=e^{\Delta t}: t \ge 0\}$ in one step.
Here and in what follows, ${\rm Id}$ denotes the identity operator in $V_N$.
Iterating \eqref{full+} for $m$-times, we obtain 
\begin{align}\label{full-sum}
X^N_j
=S_{N, \tau}^j X^N_0+\tau \sum_{i=0}^{j-1} S_{N, \tau}^{j-i} \PP_N F(X^N_{i+1})
+\sum_{i=0}^{j-1} S_{N, \tau}^{j-i} \PP_N G(X^N_i) \delta_i W, \quad j \in \nn_+.
\end{align}

To investigate the solvability of the DIEG scheme \eqref{die-g}, we need to consider the implicit operator $\hat F: V_N \to V_N$ defined by
  \begin{align} \label{hatF}
  \hat F(x)=({\rm Id}- \tau \Delta_N) x- \tau \PP_N F(x), \quad x \in V_N. 
  \end{align} 
  Then \eqref{die-g} becomes
  \begin{align} \label{hatF+}
      \hat F(X^N_{j+1}) = X^N_j + \PP_N G(X^N_j) \delta_j W, \quad j \in \nn.
  \end{align}

\begin{lemma} \label{open-dieg}
Under the condition \eqref{f-mon} with $(K_1-\lambda_1) \tau<1$, $\hat F: V_N \to V_N$ defined in \eqref{hatF} is bijective so that the DIEG scheme \eqref{die-g} can be uniquely solved pathwise.
  Morevoer, $\hat F$ is an open map, i.e. for each open set $A\in \BB(V_N)$, $\hat F(A)$ is also an open set in $\BB(V_N)$. 
\end{lemma}

\begin{proof} 
It follows from the one-sided Lipschitz condition \eqref{F-mon} and the Poincar\'e inequality \eqref{poin} that 
  \begin{align*}
      \<x-y, \hat F(x)-\hat F(y)\> 
      &=  \<x-y, ({\rm Id}-\tau  \Delta_N)(x-y)- \tau \PP_N (F(x)-F(y))\> \\
      &= \norm{x-y}^2 + \tau \norm{\nabla(x-y)}^2 -\tau ~ _{1}\<x-y,F(x)-F(y)\>_{-1} \\
      & \geq (1-(K_1-\lambda_1) \tau )\norm{x-y}^2:=C_0 \norm{x-y}^2, 
      \quad \forall~ x,y \in V_N.
  \end{align*}
As $(K_1-\lambda_1) \tau <1$, $\hat F$ defined in \eqref{hatF} is strictly monotone in the finite-dimensional Hilbert space $V_N$ and thus invertible (see, e.g., \cite[Theorem C.2]{SH96}), so that \eqref{die-g} is uniquely solved pathwise.  

It remains to show that $\hat F$ is an open map.
From the above strict monotonicity and the fact that 
  \begin{align*}
    \| \hat F(x) - \hat F(y)\| \cdot \|x-y\| \geq \<x-y, \hat F(x)-\hat F(y)\>,  \quad \forall~x,y \in V_N,
  \end{align*}
  we obtain 
  \begin{align} \label{hatF-}
    \| \hat F(x) - \hat F(y)\| \geq C_0 \|x-y\|,   \quad \forall~x,y \in V_N.
  \end{align}
  This shows that
  \begin{align}\label{ball inclusion}
    B( \hat F(x), r) \subset \hat F(B(x, r/C_0)), \quad \forall ~x \in V_N, ~r>0.
  \end{align}
%  Indeed,  if $\hat F(y)\in B_{\dot H^{-1}}( \hat F(x), r)$, then we must have $y\in B_N(x, r/C)$ so $\hat F(y)\in \hat F(B_N(x, r/C))$. 
Let us fix an open set $A\in \BB(V_N)$. Then, for each point $x \in A$, there exists an open ball $B(x, r_0/C_0) \subset A$ with $r_0>0$.
  Due to the inclusion \eqref{ball inclusion}, we have $B(\hat F(x), r_0)\subset \hat F(B(x, r_0/C_0)) \subset \hat F(A)$, which shows that $\hat F(A)$ is an open map. 
\end{proof}

\subsection{Preliminaries on Ergodicity of Markov Chains}
\label{sec2.2}

Denote by $P: V_N \times \BB(V_N) \to [0, 1]$ the transition kernel of the Markov chain $\{X_n^N: n \in \nn\}$ generated by \eqref{die-g}, i.e.,
\begin{align} \label{markov}
P(x,A)=\pp(X_{n+1}^N \in A\mid X_n^N=x), \quad x \in V_N, ~ A \in \BB(V_N).
      \end{align}
We also use $P_n$, $n \in \nn_+$, to denote the corresponding Markov semigroup on $\BB(V_N)$:
\begin{align*}
P_n \phi(x):=\ee [\phi(X_n^N)| X_0^N=x], \quad x \in V_N, ~ \phi \in \BB_b(V_N).
      \end{align*} 
      
A probability measure $\mu$ on $V_N$ is called invariant for the Markov chain $\{X_n^N: n \in \nn\}$ or its transition kernel $P$, if 
\begin{align*}
\int_{V_N} P \phi(x) \mu({\rm d}x)
=\mu(\phi):=\int_{V_N} \phi(x) \mu({\rm d}x),
\quad \forall~ \phi \in \CC_b(H).
\end{align*} 
This is equivalent to $\int_N P(x, A) \mu({\rm d}x)=\mu(A)$ for all $A \in \BB(H).$
An invariant (probability) measure $\mu$ is called ergodic for $\{X_n^N: n \in \nn\}$ or $P$, if 
\begin{align} \label{df-erg}
\lim_{m \to \infty} \frac1m \sum_{n=0}^m P_n \phi(x) 
=\mu(\phi) 
\quad \text{in}~ L^2(V_N; \mu),\quad \forall~ \phi \in L^2(V_N; \mu).
\end{align}
It is well-known that if $\{X_n^N: n \in \nn\}$ admits a unique invariant measure, then it is ergodic; in this case, we call it uniquely ergodic.

    \section{Unique Ergodicity of DIEG}
\label{sec3}

This section will show the unique ergodicity of the DIEG scheme \eqref{die-g} and then apply it to the stochastic Allen--Cahn equation \eqref{ac} driven by nondegenerate multiplicative white noise.

\subsection{Lyapunov structure of DIEG}

We begin with the Lyapunov structure of the DIEG scheme \eqref{die-g}.
To this end, we need the following uniform estimation for the sum $(W^N_j)_{j \in \nn_+}$ of discrete stochastic convolutions in \eqref{full-sum}, where we define 
\begin{align}\label{ou}
W^N_j:=\sum_{i=0}^{j-1} S_{N, \tau}^{j-i} \PP_N G(X^N_i) \delta_i W,
\quad j \in \nn_+.
\end{align} 

 \begin{lm} \label{lm-ou}
      Let \eqref{f-mon} and \eqref{g-bou} hold with $(K_1-\lambda_1) \tau<1$.
Then for any $p \ge 2$, $N \in \nn_+$, $\tau \in (0, 1)$ with $(K_1-\lambda_1) \tau <1$, and $\beta \in [0, 1/2)$, there exists a positive constant $C=C(p, \beta, K_6)$ such that  
\begin{align}\label{err-ou}
\sup_{N \in \nn_+} \sup_{j \in \nn_+} \ee \|W^N_j\|_\beta^p \le C.
      \end{align} 
  \end{lm}

  \begin{proof}
  Denote by $Z_i:=(-\Delta)^{\beta/2}S_{N, \tau}^{j-i} \PP_N G(X^N_i) \delta_i W$, $i \in \nn$. For any $i \in \nn$, it is clear that $G(X^N_i)$ is $\FFF_{t_i}$-measurable and that $\delta_i W$ is $\FFF_{t_{i+1}}$-measurable and independent of $\FFF_{t_i}$.
  Consequently, $\{Y_j:=\sum_{i=0}^{j-1} Z_i, j \in \nn_+; ~ Y_0:=0\}$ is a ($V_N$-valued) $(\FFF_{t_j})_{j \in \nn_+}$-discrete martingale, so that $\{Z_i=Y_{i+1}-Y_i: i \in \nn\}$ is an $(\FFF_{t_j})_{j \in \nn_+}$-martingale difference.

Using the discrete Burkholder--Davis--Gundy (BDG) inequality (see, e.g., \cite[Lemma 2.2]{LQ21}), we obtain 
  \begin{align*} 
\ee \|W^N_j\|_\beta^p 
& =\ee \|\sum_{i=0}^{j-1} Z_i\|^p
\le C (\sum_{i=0}^{j-1} \|Z_i\|^2_{L^p_\omega L_x^2})^{p/2} \\
& = C \Big(\sum_{i=0}^{j-1} \Big\|\int_{t_i}^{t_{i+1}}(-\Delta)^{\beta/2}S_{N, \tau}^{j-i} \PP_N G(X^N_i) dW_t\Big\|^2_{L^p_\omega L_x^2} \Big)^{p/2}.
\end{align*}  
Here and in the rest of the paper, $C$ denotes a universally positive constant that would differ in each appearance, and $L^p_\omega$ and $L_x^2$ denote the usual $L^p$- or $L^2$-norm in $\Omega$ and $(0, 1)$, respectively.
Then, we use the continuous BDG inequality, the definition of the $\LL_2^0$-norm, the relation $\eqref{Delta}$, and the properties of $S_{N, \tau}$ and $\PP_N$ to get 
\begin{align*} 
\ee \|W^N_j\|_\beta^p 
& \le C \tau^{p/2} \Big(\sum_{i=0}^{j-1}\|(-\Delta)^{\beta/2}S_{N, \tau}^{j-i} \PP_N G(X^N_i)\|^2_{L^p_\omega \LL_2^0} \Big)^{p/2} \\
& = C \tau^{p/2} \Big(\sum_{i=0}^{j-1} \Big\|\sum_{m, n} \<(-\Delta)^{\beta/2}S_{N, \tau}^{j-i} \PP_N G(X^N_i) e_m, e_n\>^2\Big\|_{L^{p/2}_\omega} \Big)^{p/2} \\ 
& = C \tau^{p/2} \Big(\sum_{i=0}^{j-1} \Big\|\sum_n \lambda_n^\beta (1+\tau \lambda_n)^{-2(j-i)} \sum_m  \<G(X^N_i) e_m, e_n\>^2\Big\|_{L^{p/2}_\omega} \Big)^{p/2} \\
& = C \tau^{p/2} \Big(\sum_{i=0}^{j-1} \Big\|\sum_n \lambda_n^\beta (1+\tau \lambda_n)^{-2(j-i)} \|G(X^N_i) e_n\|^2 \Big\|_{L^{p/2}_\omega} \Big)^{p/2}.
\end{align*}  
It follows from the condition \eqref{g-bou} that  
\begin{align*} 
\ee \|W^N_j\|_\beta^p 
& \le C \tau^{p/2} \Big(\sum_n \lambda_n^\beta \sum_{i=0}^{j-1} (1+\tau \lambda_n)^{-2(j-i)} \Big)^{p/2}  
\le C \Big(\sum_n \lambda_n^{\beta-1} \Big)^{p/2}, 
\end{align*}  
which is finite as $\beta<1/2$, where in the last inequality, we have used the elementary identity
$\sum_{i=0}^{j-1} (1+\tau \lambda_n)^{-2(j-i)}=[\tau \lambda_n(\lambda_n+2)]^{-1}$.
  \end{proof}
  
  Now, we can develop the following Lyapunov structure for the DIEG scheme \eqref{die-g}.

 \begin{prop} \label{prop-lya}
      Let Assumption \ref{ap-f} and condition \eqref{g-bou} hold with $K_2<\lambda_1$.
Then for any $N \in \nn_+$, $\tau \in (0, 1)$ with $(K_1-\lambda_1) \tau <1$, and $\FFF_0$-measurable $X_0^N \in L^2(\Omega; V_N)$, there exist positive constants $\gamma$ and $C_\gamma$ such that
      \begin{align}\label{lya}
& \ee \|X_j^N\|^2 \le e^{-\gamma t_j} \ee \|X_0^N\|^2+C_\gamma, \quad j \in \nn.
      \end{align} 
  \end{prop}

\begin{proof}  
For $j \in \nn$, set $Y^N_j:=X^N_j-W^N_j$, where $W^N_j$ is defined in \eqref{ou}.
Then from \eqref{full-sum} we have 
\begin{align*}
Y^N_{j+1}
=S_{N, \tau}^{j+1} X^N_0+\tau \sum_{i=0}^j S_{N, \tau}^{j+1-i} \PP_N F(Y^N_{i+1}+W^N_{i+1}).
\end{align*}
It is clear that 
\begin{align}\label{y} 
Y^N_{j+1}
=Y^N_j+\tau \Delta_N Y^N_{j+1}
+\tau \PP_N F(Y^N_{j+1}+W^N_{j+1}).  
\end{align} 

Testing \eqref{y} with $Y^N_{j+1}$ under the $\<\cdot, \cdot\>$-inner product, using the elementary equality 
\begin{align*}
2 \<x-y, x\> =\|x \|^2- \|y\|^2+\|x-y\|^2,
\quad x, y \in V_N,
\end{align*}  
and integration by parts formula, we have  
\begin{align*}
& \|Y^N_{j+1}\|^2 - \|Y^N_j\|^2 +  \|Y^N_{j+1}-Y^N_j\|^2
+ 2 \tau \|\nabla Y^N_{j+1}\|^2   \\
& =2 ~ \<Y^N_{j+1}, F(X^N_{j+1})-F(W^N_{j+1})\> \tau  
 + 2 \<Y^N_{j+1}, F(W^N_{j+1})\> \tau.
\end{align*}   
Using the condition \eqref{F-coe} and Cauchy--Schwarz inequality leads to
\begin{align*}
& \|Y^N_{j+1}\|^2 - \|Y^N_j\|^2 + 2 \tau \|\nabla Y^N_{j+1}\|^2  \\
& \le 2 K_2 \tau \|Y^N_{j+1}\|^2  
 + 2 \tau~ \<Y^N_{j+1}, F(W^N_{j+1})\>\\
 & \le 2 K_2 \tau \|Y^N_{j+1}\|^2  + 2 \varepsilon \tau \|\nabla Y^N_{j+1}\|^2
 + C_\varepsilon \tau \|F(W^N_{j+1})\|_{-1}^2,
\end{align*}   
for any positive $\varepsilon$.
Taking the expectation on both sides and using the Poincar\'e inequality \eqref{poin}, we obtain  
\begin{align*}  
\ee \|Y^N_{j+1}\|^2  
& \le \frac1{1+[(2-\varepsilon) \lambda_1-2 K_2] \tau} \ee \|Y^N_j\|^2 \\
& \quad + \frac{C_\varepsilon \tau}{1+[(2-\varepsilon) \lambda_1-2 K_2] \tau} 
\ee \|F(W^N_{j+1})\|_{-1}^2.
\end{align*}   
By the embeddings $L^1 \subset \dot H^{-1}$ and $\dot H^\beta \subset \dot L^q$ for sufficiently large $\beta<1/2$, the growth condition \eqref{f-gro}, and the estimation \eqref{err-ou}, we have 
\begin{align*} 
\ee[\|F(W^N_j)\|_{-1}^2]  
& \le C \ee[\|F(W^N_j)\|_{L^1}^2] 
\le C (1+\ee[ \||W^N_j\|_{L^q}^{2q}]) \\
& \le C (1+\ee[ \||W^N_j\|_\beta^{2q}]) \le C.
\end{align*}  
Combining the above two estimates leads to
\begin{align*} 
& \ee \|Y^N_{j+1}\|^2  
\le \frac1{1+[(2-\varepsilon) \lambda_1-2 K_2] \tau} \ee \|Y^N_j\|^2
+ \frac{C_\varepsilon \tau}{1+[(2-\varepsilon) \lambda_1-2 K_2] \tau}.
\end{align*} 
from which we obtain 
\begin{align*} 
\ee \|Y^N_j\|^2  
& \le \Big(\frac1{1+[(2-\varepsilon) \lambda_1-2 K_2] \tau}\Big)^j \ee \|X_0\|^2 \\
& \quad + \frac{C_\varepsilon \tau}{1+[(2-\varepsilon) \lambda_1-2 K_2] \tau}
\sum_{i=0}^{j-1} \Big(\frac1{1+[(2-\varepsilon) \lambda_1-2 K_2] \tau}\Big)^i.
\end{align*}         
Note that 
$a^k < e^{-(1-a) k}$ for any $a \in (0, 1)$ 
and that $\frac1{1+[(2-\varepsilon) \lambda_1-2 K_2] \tau}<1$ for any $\tau \in (0, 1)$, which is ensured by the condition $K_2<\lambda_1$ and the fact that $\varepsilon$ can be taken as an arbitrary small positive constant, we conclude \eqref{lya} with 
$\gamma:=\frac{(2-\varepsilon) \lambda_1-2 K_2}{1+[(2-\varepsilon) \lambda_1-2 K_2] \tau}$ and 
$C_\gamma:=\frac{C_\varepsilon}{(2-\varepsilon) \lambda_1-2 K_2}$.
\end{proof}

\begin{rk}
The result in Proposition \ref{prop-lya} indicates that $V: V_N \to [0,\infty)$ defined by $V(x)=\|x\|^2$, $x \in V_N$, is a Lyapunov function of the DIEG scheme \eqref{die-g}. 
  \end{rk}

  \begin{rk}
  We impose the bounded assumption on the diffusion coefficient in the present white noise case. 
  For an unbounded diffusion coefficient, such stochastic-random transform argument would fail as one could not derive the estimate \eqref{err-ou} for  $W^N_j$ defined in \eqref{ou}.  
  \end{rk}

\subsection{Regularity Property}
\label{sec4}

In this part, we aim to derive the regularity property of the transition kernel $P$ defined in \eqref{markov} associated with the DIEG scheme \eqref{die-g} in the sense that all transition probabilities of \eqref{die-g} are equivalent.

\begin{prop}\label{prop-uni}
      Let Assumptions \ref{ap-f} and \ref{ap-g} hold. 
Then for any $N \in \nn_+$ and $\tau \in (0, 1)$ with $(K_1-\lambda_1) \tau <1$, $P$ is regular. 
Consequently, there exists at most, if it exists, one invariant measure of $\{X_n^N: n \in \nn\}$ in $V_N$. 
  \end{prop}

  \begin{proof}
        Let $x \in V_N$ and $A$ be a non-empty Borel open set in $V_N$.
By \eqref{hatF+}, we have 
      \begin{align} \label{pxa+}
          P(x, A) 
          & =\pp(X_{n+1}^N\in A \mid X_n^N=x)
          = \mu_{x, [\PP_N G(x)] [\PP_N G(x)]^\tr \tau}(\hat F(A)),
      \end{align} 
      as $x +\PP_N G(x) \delta_n W \sim N(x, [\PP_N G(x)] [\PP_N G(x)]^\tr \tau)$, 
      where $\mu_{a, b}$ denotes the Gaussian measure in $V_N$ with mean $a \in V_N$ and variance operator $b \in \LL(V_N)$. 
It was shown in Lemma \ref{open-dieg} that $\hat F$ is an open map so that $\hat F(A)$ is a non-empty open set.
Due to the non-degeneracy of $G$ in Assumption \ref{ap-g}, the Gaussian measure $\mu_{x, [\PP_N G(x)] [\PP_N G(x)]^\tr \tau}(\hat F(A))$ is non-degenerate in $\BB(V_N)$.
Indeed, for any $v_N \in V_N \setminus \{0\}$,
\begin{align*}
\<[\PP_N G(x)] [\PP_N G(x)]^\tr v_N, v_N\>
& =\|[\PP_N G(x)]^\tr v_N\|^2
=\sum_{m \in \nn_+} \<[\PP_N G(x)]^\tr v_N, e_m\>^2 \\
& =\sum_{m \in \nn_+} \<v_N, \PP_N G(x) e_m\>^2
=\sum_{m \in \nn_+} \<v_N, G(x) e_m\>^2 \\
% & =\sum_{m \in \nn_+} \Big(\int_0^1 v_N(\xi) g(x(\xi)) e_m(\xi) d\xi\Big)^2  \\
& =\int_0^1 |g(x(\xi))|^2 |v_N(\xi)|^2 d\xi>0,
\end{align*}
as $g$ is invertible in $\rr$. 
It is known that any nondegenerate Gaussian measure in separable Banach space measures any non-empty open set positive, so the open set $\hat F(A) \in \BB(V_N)$ has positive measure under $\mu_{x, [\PP_N G(x)] [\PP_N G(x)]^\tr \tau}$, which implies $P(x, A)>0$ and shows the irreducibility of $\{X_n^N: n \in \nn\}$ in $V_N$.
It is well-known that all non-degenerate Gaussian measures are equivalent in the finite-dimensional case.
Thus, $P$ is regular, and by Doob theorem, it possesses at most one invariant measure. 
  \end{proof}

Now, we can show the unique ergodicity of the DIEG scheme \eqref{die-g}.

  \begin{theorem} \label{tm-spde}
      Let Assumptions \ref{ap-f} and \ref{ap-g} hold with $K_2<\lambda_1$.
Then \eqref{die-g} is uniquely ergodic for any $N \in \nn_+$ and $\tau \in (0, 1)$ with $(K_1-\lambda_1) \tau<1$.
  \end{theorem}

  \begin{proof} 
The Lyapunov condition \eqref{lya} in Theorem \ref{prop-lya}, together with the Feller property followed from the regularity property in Proposition \ref{prop-uni}, imply the existence of an invariant measure for $\{X_n^N: n \in \nn\}$.
Combined with the regularity property in Proposition \ref{prop-uni}, we conclude the uniqueness of the invariant measure for $\{X_n^N: n \in \nn\}$.
  \end{proof}

Applying the above result of Theorem \ref{tm-spde}, we have the following unique ergodicity of the DIEG scheme \eqref{die-g} applied to the stochastic Allen--Cahn equation \eqref{ac}.

  \begin{theorem} \label{tm-ac}
  Let Assumption \ref{ap-g} hold.
  For any $\epsilon>0$, $N \in \nn_+$ and $\tau \in (0, 1)$ with $(\epsilon^{-2}-\lambda_1) \tau<1$, \eqref{die-g} applied to Eq. \eqref{ac} is uniquely ergodic.
\qed 
  \end{theorem}

  \begin{proof}  
  We just need to check that the conditions in Assumption \ref{ap-f} hold with $K_2<\lambda_1$ in the setting of Eq. \eqref{ac} with $q=3$, $f(\xi)=\epsilon^{-2}(\xi-\xi^3)$, $\xi \in \rr$, and $g$ satisfying Assumption \ref{ap-g}.
   
  The validity of \eqref{f-mon} and \eqref{f-gro} in Assumption \ref{ap-f} in the setting of Eq. \eqref{ac} were shown in \cite[Section 4]{Liu23}: $K_1=\epsilon^{-2}$, $K_4=2\epsilon^{-2}$, and $K_5=\epsilon^{-2}$. 
Morevoer, one can take $K_2$ to be any negative scalar and thus \eqref{f-coe} with $K_2<\lambda_1$ and some $K_3>0$; see \cite[Theorem 4.4]{LL24}.
  \end{proof}

\section{Numerical Experiments}
\label{sec4}

  In this section, we perform some numerical experiments to verify our theoretical result, Theorem \ref{tm-ac} in Section \ref{sec3}.
 
The numerical test is given to the stochastic Allen--Cahn equation \eqref{ac} in $\OOO=(0, 1)$ with $\epsilon=0.5$ and $g(x)=2+\sin x^2$.
By Theorem \ref{tm-ac}, the DIEG scheme \eqref{die-g} applied to Eq. \eqref{ac} is uniquely ergodic for any $\tau \in (0, 1)$ (fulfilling the conditon $(\epsilon^{-2}-\lambda_1) \tau<1$). 
We take $\tau=0.05$ and $N=10$ (the dimension of the spectral Galerkin approximate space), choose three initial data $X_0(\xi)=\sin \pi \xi, \sum_{k=1}^{10} \sin k \pi \xi, -\sum_{k=1}^{10} \sin k \pi \xi$, $\xi \in (0, 1)$, respectively, and approximate the expectation by taking averaged value over $1,000$ paths to implement the numerical experiments.
In addition, we simulate the time averages $\frac{1}{2,000}\sum_{n=1}^{2,000} \ee [\phi(X_n^N)]$ (up to $n=2,000$ corresponding to $t=100$) by 
  \begin{align*}
    \frac{1}{2,000,000}\sum_{n=1}^{2,000} \sum_{k=1}^{1,000} \phi(X_n^{N, k}),
  \end{align*}
  where $X_n^{N, k}$ denotes $n$-th iteration of $k$-th sample path and the test function $\phi$ are chosen to be $\phi(\cdot)=e^{-\|\cdot\|^2}, \sin \|\cdot\|^2, \|\cdot\|^2$, respectively.

  \begin{figure}[h]
    \centering
    \includegraphics[width=1\textwidth]{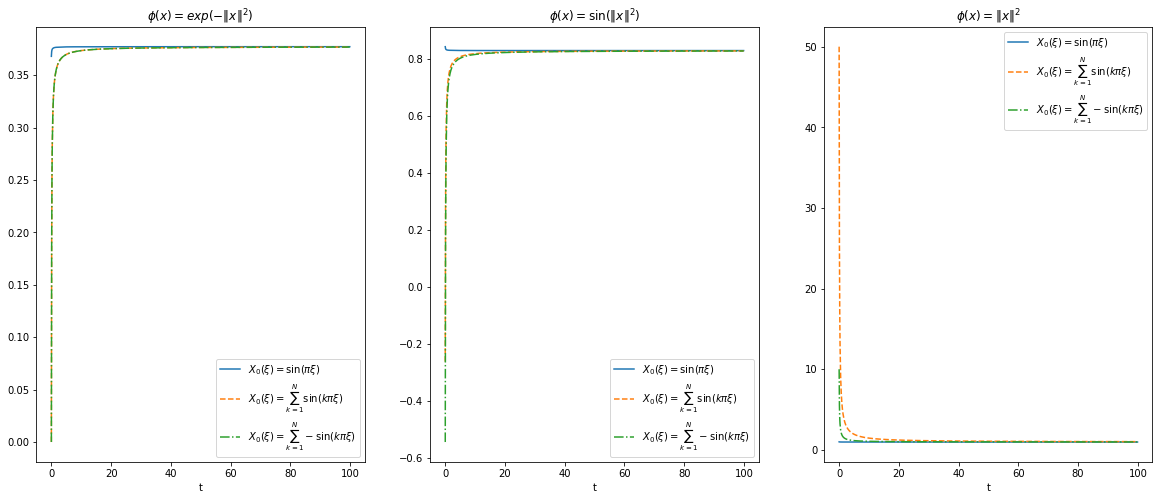} 
    \caption{Time averages of \eqref{die-g} for Eq. \eqref{ac}}\label{fig-ac}
  \end{figure}

From Figure \ref{fig-ac}, the time averages of the DIEG scheme \eqref{die-g} with different initial data converge to the same ergodic limit, which coincides with the theoretical result in Theorem \ref{tm-ac}.
This experiment also indicates that the original 1D stochastic Allen--Cahn equation \eqref{ac}, with bounded, invertible, and continuous diffusion coefficient, driven by multiplicative white noise, is uniquely ergodic, which will be investigated in a separate paper.

  \bibliographystyle{plain}
  \bibliography{bib.bib}

\end{document}